\documentclass[a4paper]{article}
\usepackage{amsmath,amsthm,amssymb,mathrsfs}
\usepackage{enumerate}
\usepackage[dvipdfm]{pict2e}
\usepackage{bbm,cases}
\usepackage{color}

\newtheorem{dfn}{Definition}[section]
\newtheorem{thm}[dfn]{Theorem}

\newtheorem{lem}[dfn]{Lemma}
\newtheorem{cor}[dfn]{Corollary}
\newtheorem{prop}[dfn]{Proposition}\makeatletter

             \@addtoreset{equation}{section}
\makeatother

\newcommand{\dis}{\displaystyle}
\newcommand{\ve}{\varepsilon}
\allowdisplaybreaks[4]
\begin{document}
\title{Nonnegative solutions to stochastic heat equation with nonlinear drift}
\author{Makoto Nakashima \footnote{nakamako@math.tsukuba.ac.jp, Division of Mathematics, Graduate School of Pure and Applied Sciences,University of Tsukuba, 1-1-1 Ten-noudai, Tsukuba-shi, Ibaraki-ken, Japan } }
\date{}
\maketitle

\begin{abstract}
In this paper, we consider one-dimensional stochastic heat equation with nonlinear drift, 
$\dis \partial_t u=\frac{1}{2}\Delta u+b(u)u+\sigma(u)\dot{W}(t,x)$, where $b:\mathbb{R}_{+}\to \mathbb{R}$ is a continuous function and $\sigma:\mathbb{R}_{+}\to \mathbb{R}$ is a continuous function with suitable property.
We will construct nonnegative solutions to such SPDEs.

\end{abstract}

\vspace{1em}
We denote by  $(\Omega, {\cal F},P )$ a probability space. Let $\mathbb{N}=\{0,1,2,\cdots\}$, $\mathbb{N}^*=\{1,2,3,\cdots\}$, and $\mathbb{Z}=\{0,\pm 1,\pm 2,\cdots\}$. Let $C(S)$ be the set of continuous functions on $S$.

\section{Introduction and Main Result}

We consider the nonnegative solutions $u(t,x)$ with $t\geq 0$, $x\in\mathbb{R}$, to the stochastic heat equation\begin{align}
\partial _tu&=\frac{1}{2}\Delta u+a(u)+\sigma(u)\dot{W}\notag\\
u(0,x)&=u_0(x),\label{SPDE}\tag{SPDE$_{a\sigma}$}
\end{align}  
where $\dot{W}=\dot{W}(t,x)$ is $2$-parameter white noise. This type of stochastic heat equations appears in several models related to physics or population genetics. 

\vspace{1em}
{\bf Example 1}  If $a(x,u)=a(x)u$, $a(x)$ is a bounded continuous function in $x$ and  $\sigma(u)=\sqrt{u}$, then a solution to (\ref{SPDE}) corresponds to the density $u(t,x)dx=X_t(dx)$, where $X_t$ is the one-dimensional super-Brownian motion\cite{KonShi,Rei}. 

{\bf Example 2} If $a(u)=p(1-u)+qu+ru(1-u)$ for $p,q\geq 0$ and $r\in \mathbb{R}$, $\sigma(u)=\sqrt{u(1-u)}$ and $u_0\in [0,1]$, then the solution to (\ref{SPDE}) corresponds to  the density for the scaling limit  of the stepping-stone model\cite{Shi2}.

{\bf Example 3} If $a(u)=\theta u-u^2$ for $\theta\geq 0$ and $\sigma(u)=\sqrt{u}$, then a  solution to (\ref{SPDE}) arises as the density of the limit of the long-range contact process and voter model\cite{MueTri}.

{\bf  Example 4} If $a(u)\equiv 0$ and $\sigma(u)=u$, then the solution to (\ref{SPDE}) is the Cole-Hopf solution to KPZ equation arising at the statistical mechanics \cite{BerGia}. It is known that the solution is pathwise unique\cite{Wal}.

{\bf Example 5} If $a(u)\equiv 0$ and $\sigma(u)=\sqrt{u+u^2}$, then a solution to (\ref{SPDE}) appears as a density of a weak  limit process of some branching random walks in random environment\cite{Nak}. Its ``dual process" is also a solution to (\ref{SPDE}) for $a(u)=-\frac{1}{2}u^2$ and  $\sigma(u)=u$.

{\bf Remark:} The uniqueness in law has been already known for the above examples under some initial condition.

Also, the existence of the solutions to such SPDE has been studied well. Iwata showed the existence and the uniqueness in law of the case where $a(u)$ and $\sigma(u)$ are global Lipschitz
continuous, or $\sigma(u)$ is bounded and $a(u)$ grows at most polynomially with some condition\cite{Iwa}.  Mueller and Perkins showed the existence for the case where  $a(u)=0$ and $\sigma(u)$ is a general continuous function with some growth condition, and showed compact support property of the solutions\cite{MuePer}. 

In this paper, we will prove the existence of nonnegative solutions with local Lipschitz continuity on $a(u)$ with some condition and without boundedness of $\sigma(u)$.

\vspace{1em}

To state our main theorem, we introduce some notations.

In this paper, we suppose that $a(u)=b(u)u$ for some continuous function $b:\mathbb{R}_+\to \mathbb{R}$.

Let $C_b^n(\mathbb{R})$, $C_c^{n}(\mathbb{R})$ and $C_0^n(\mathbb{R})$ be the set of $n$-th continuously differentiable  functions with bounded, compact support and vanishing at infinity, respectively. Also, the subscript `$+$' means the subset of the nonnegative elements. 

\begin{dfn}({\bf Rapidly decreasing functions})
For $C(\mathbb{R})$, let \begin{align*}
|f|_p=\sup_x\left|	e^{p|x|}f(x)		\right|,\ \ p\in\mathbb{R}.
\end{align*}
We introduce subspace $C_{rap}(\mathbb{R})$ of $C(\mathbb{R})$ by\begin{align*}
C_{rap}(\mathbb{R})=\left\{	f\in C(\mathbb{R}): |f|_\lambda<\infty\ \ \text{for every }\lambda>0		\right\}.
\end{align*}

\end{dfn}


Let $\theta>0$, $r\in\left(0,1\right]$ and $L_b,l_b,L_\sigma\geq 0$.

\begin{thm}\label{main}
Let $b\in C({\mathbb{R}_+})$ and $\sigma \in C(\mathbb{R}_+)_+$ satisfies \begin{align}
-l_b(u^\theta+1 )&\leq b(u)\leq L_b\label{b}\\
0&\leq \sigma(u)\leq L_\sigma(u^r+u) .\label{sigma}
\end{align}
Then, for any $u_0\in C_{rap}(\mathbb{R})_+$, there are solutions to the following martingale problem:\begin{align}
\begin{cases}
\text{For all }\phi \in C_b^2(\mathbb{R})\\
Z_t(\phi)=\int \phi(x)u(t,x)dx-\int \phi(x)u_0(x)dx\\
\hspace{5em}-\int_0^t\int (\frac{1}{2}\Delta +b(u(s,x)))\phi(x)u(s,x)dxds	\\
\text{is an }{\cal F}_t^u\text{-martingale with }\\
\langle		Z(\phi)	\rangle_t=\int_0^t	\int \sigma(u(s,x))^2\phi(x)^2dxds.	
\end{cases}\label{martiag}
\end{align}
Especially, $t\mapsto u(t,\cdot)$ is the continuous map from $\mathbb{R}_+\to C_{rap}^+(\mathbb{R})$.
\end{thm}

{\bf Remark:} Solutions to the martingale problem of (\ref{martiag}) are solutions to (\ref{SPDE})

{\bf Remark:} If $b(u)\vee 0$ is unbounded and solutions exist, then solutions may blow up in finite time. Actually, $b(u)=u^\alpha$ for $\alpha>0$ with boundary condition $u(t,0)=u(t,R)=0$ for  $t\geq 0$, $R>0$, then the solution blow up\cite{FilKoh,Fuj,LeeNi}.

\section{Preliminary}

Let $B_t$ be the one dimensional Brownian motion and $P_x$ be the law of $B$ starting at $x$. 
Also, we denote by $P_{s,x}$ the law of $Y_t=(t,B_t)$ for $0\leq s\leq t$, $x\in\mathbb{R}$. 
Let ${\cal E}_b$ be the set of the bounded measurable functions on $\mathbb{R}_+\times \mathbb{R}$.

For $\phi\in {\cal E}_b$, we define the semigroup $P_t$ associated to $(t,B_t)$ by \begin{align*}
P_t\phi(s,x)=P_{s,x}\left(	\phi(t,B_t)		\right).
\end{align*}
Also, we denote by $A$ the generator of $Y$, that is for $\phi\in C_b(\mathbb{R}_+\times \mathbb{R})$\begin{align*}
A\phi(s,x)=\lim_{h\searrow 0}\frac{P_{s,x}\left(		\phi(Y_{s+h})		\right)-\phi(s,x)}{h},\ \ \text{if the limit exists,}
\intertext{and }
{\cal D}(A)=\left\{	\phi\in {\cal E}_b:	\lim_{h\searrow 0}\frac{P_{s,x}\left(		\phi(Y_{s+h})		\right)-\phi(s,x)}{h}\text{ exists.}		\right\}
\end{align*}
In  particular, $\phi(t,s)=\phi_1(t)\phi_2(x)$ for $\phi_1\in C_b^1(\mathbb{R})$, and $\phi_2\in C_b^2(\mathbb{R})$, \begin{align*}
A\phi(t,x)=\left(\frac{\partial }{\partial t}+\frac{1}{2}\Delta\right) \phi
\end{align*}

Let ${\cal M}_F(\mathbb{R})$ be the space of finite measures on $\mathbb{R}$ with the topology of weak convergence.

Let $b\in C_b(\mathbb{R})$ and $\gamma\in C_b(\mathbb{R})_+$. Then, for $t\in [0,\infty)$ and $m\in {\cal M}_F(\mathbb{R})$, there exists Dawson-Watanabe superprocess characterized by the unique solution to the following martingale problem:\begin{align*}
\begin{cases}
\text{For any }\phi\in {\cal D}(A),\\
\quad Z_t(\phi)=X_t(\phi)-m(\phi)-\int_0^t	\left(A+b\right)\phi (s,x)X_s(dx)ds	\\
\text{is an ${\cal F}_t^X$-martingale and }\\
\quad \langle	Z(\phi)	\rangle_t=\int_0^t\gamma(x)\phi(s,x)^2X_s(dx)ds	
\end{cases}
\end{align*} 
We denote it by ${X}$ and its law by ${\mathbb{P}}_{m,b,\gamma}$. Especially, we call it $(B,b,\gamma)$-super-Brownian motion. We remark that ${X}$ takes continuous ${\cal M}_F(\mathbb{R})$-valued paths.
We define \begin{align*}
&\Omega_X=C\left(	[0,\infty),{\cal M}_F(\mathbb{R})\right)
\end{align*}
with its Borel $\sigma$-field ${\cal F}^X$. 
Let \begin{align*}
{\cal F}^X[s,t]=\sigma\left(X_r:s\leq r\leq t\right)
\intertext{and}
{\cal F}^X[s,t+]=\bigcap_{n=1}^\infty {\cal F}^X\left[s,t+\frac{1}{n}\right].
\end{align*}
Especially, we write ${\cal F}_t^X={\cal F}^X[0,t]$.

For our convenience, we will identify the element $u$ of $L^1_+(\mathbb{R})$ with the element of ${\cal M}_F(\mathbb{R})$ by \begin{align*}
\int_B u(x)dx,\ \ \text{for }B\in {\cal B}(\mathbb{R}).
\end{align*}

\section{Proof}
In this section, we will construct a solution of the martingale problem (\ref{martiag}). 

At the moment, we assume that \begin{align*}
r\in\left(\frac{1}{2},1			\right]\end{align*} and we define $\gamma:\mathbb{R}_+\to \mathbb{R}$ by\begin{align*}
\gamma(u)=\frac{\sigma^2(u)}{u}1\{u>0\}.
\end{align*}
We remark that $\gamma (u)$ is continuous in $u\in [0,\infty)$ and $0\leq \gamma(u)\leq C(\gamma)(1+u)$.

For $X\in \Omega_X$, we define \begin{align*}
u(t,x)=\begin{cases}
\lim_{n\to \infty}2^nX_t( I_n(x)),\ \ &\text{if the limit exists},\\
0&\text{otherwise},
\end{cases}
\end{align*}
where $I_n(x)=[j2^{-n},(j+1)2^{-n})\ni x$, $j\in\mathbb{Z}$. By \cite[Theorem 1.4]{KonShi} and \cite[Theorem 2.5]{Shi}, if we assume $b\in C_b(\mathbb{R})$, $\gamma \in C_b(\mathbb{R})_+$ and $m$ has rapidly decreasing continuous density $u_0$, then $u(t,x)$ exists $\mathbb{P}_{m, b,\gamma}$-a.s. and also $u(t,\cdot)\in C_{rap}(\mathbb{R})_+$.
Now, we construct new probability measures $\mathbb{P}^n$ on $(\Omega_X,{\cal F}^X,{\cal F}^X_{t})$ by induction as  follows:

 Let $\delta_n=\frac{1}{n}$ for each $n\in \mathbb{N}$. 
The restriction of $\mathbb{P}^n$ to ${\cal F}^X[0,\delta_n]$ is given by $\mathbb{P}_{u_0,b(u_0),\gamma(u_0)}$,
where we remark that $b(u_0)\in C_b(\mathbb{R})$ and $\gamma(u_0)\in C_b(\mathbb{R})_+$ since $u_0\in C_{rap}(\mathbb{R})_+$. 
That is, $X_t$ is the super-Brownian motion $(B,b(u_0),\gamma(u_0))$ up to time $t=\delta_n$.

Given $\{u(t,x):(t,x)\in[0,\delta_n]\times \mathbb{R}\}$, we suppose that the restriction of $\mathbb{P}^n$ to ${\cal F}^X[\delta_n,2\delta_n]$ is given by $\dis \mathbb{P}_{u_{\delta_n},b(u_{\delta_n}),\gamma(u_{\delta_n})}$. That is given $\{u(\delta_n,x):x\in \mathbb{R}\}$, $X_{t+\delta_n}$ evolves as the super-Brownian motion $(B,b(u_{\delta_n}),\gamma(u_{\delta_n}))$ for $0\leq t\leq \delta_n$ with staring $u_{\delta_n}(x)dx$, where we write $u(t,x)=u_t(x)$. 

Inductively, we can construct $\mathbb{P}^n$ on $(\Omega_X,{\cal F}^X,{\cal F}_t^X)$. In particular, $t\mapsto u(t,\cdot)$ is a continuous $C_{rap}(\mathbb{R})_+$-valued process under $\mathbb{P}^n$.

Let $u^n(t,x)=u([\frac{t}{\delta_n}]\delta_n,x)$. Then, for all $\phi\in {\cal D}(A)$, \begin{align*}
Z_t(\phi)&=X_t(\phi_t)-\int \phi_0(x)u_0(x)dx\\
&\hspace{2em}-\int_0^t\int\left(A \phi(s,y)+b(u^n(s,y)\right)\phi(s,y))X_s(dy)ds
\intertext{is a ${\cal F}^X_{t+}$-martingale under $\mathbb{P}^n$ and }
\langle	Z(\phi)		\rangle_t&=\int_0^t\int		\gamma(u^n(s,y))\phi^2(s,y)X_s(dy)ds,\ \mathbb{P}^n\text{-a.s..}
\end{align*}

We define the Feynman-Kac semigroup associated the Brownian motion $B$ and ${g}\in C_b({\mathbb{R}})$ by \begin{align*}
P_{t}^{{g}}\phi(s,y)=E_{s,y}\left[	\exp\left(	\int_s^t	{g}(B_r)dr	\right)	\phi(t,B_t)				\right],\ \ 0\leq s\leq t,\ \phi\in {\cal E}_b(A).
\end{align*}
Then, we have that for $\phi\in {\cal D}_b(A)$\begin{align*}
&A P_{t}^{{g}}\phi(s,y)\\
&=\lim_{h\searrow 0}\frac{1}{h}\left(E_{s,y}\left[	E_{s+h,B_{s+h}}	\left[		\exp\left(	\int_{s+h}^{t}	{g}(B_r)dr	\right)	\phi(t,B_{t})					\right]	\right]\right.\\
&\left. \hspace{10em}-E_{s,y}\left[		\exp\left(	\int_s^t	{g}(B_r)dr	\right)	\phi(t,B_t)						\right]\right)\\
&=-g(y)P_{t}^{{g}}	\left(\phi		\right)(s,y).
\end{align*}

\begin{lem}\label{semigroup}
For all $t\geq 0$ and $\phi\in {\cal D}_b(A)$, \begin{align*}
\mathbb{P}^n(|X_t(\phi)|)\leq \int P_{t}^{L_b}|\phi|(0,y)u_0(y)dy.
\end{align*}
\end{lem}
\begin{proof}First, we let $\phi\geq 0$ for all $(t,x)\in\mathbb{R}_+\times \mathbb{R}$.
Since $b(u^n)\leq L_b$ $\mathbb{P}^n$-a.s., we have that \begin{align*}
X_t(\phi_t)&=\int P_{t}^{L_b}\phi(0,y)u_0(y)dy+\int_{0}^t		\int(b(u_0(y))-L_b)P_{t}^{L_b}\phi(r,y)X_r(dy)	dr	\\
&+Z_t(P_{t}^{L_b}\phi)\\
&\leq \int P_{t}^{L_b}\phi(0,y)u_0(y)dy+Z(P_{t}^{L_b}\phi),
\end{align*}
for $t\in[0,\delta_n]$ $\mathbb{P}^n$-a.s.\,and the statement follows up to $t=\delta_n$ since $\mathbb{P}^n\left(Z\left(	P_{t}^{L_b}(\phi)	\right)\right)=0$. By definition of $\mathbb{P}^n$, we have that \begin{align*}
X_t(\phi_t)&=\int P_{t}^{L_b}\phi(\delta_n,y)u^n(\delta_n,y)dy+Z_t(P_{t}^{L_b}\phi)-Z_{\delta_n}(P_{t}^{L_b}\phi)\\
&+\int_{\delta_n}^t		\int(b(u^n(\delta_n,y))-L_b)P_{t}^{L_b}\phi(r,y)X_r(dy)	dr	\\
&\leq \int P_{t}^{L_b}\phi(\delta_n,y)u^n(\delta_n,y)dy+Z_t(P_{t}^{L_b}\phi)-Z_{\delta_n}(P_{t}^{L_b}\phi)
\end{align*}
for $t\in [0,2\delta_n]$ and 	by the Markov property, $\dis \mathbb{P}^n\left(	\left.Z_t(P_{t}^{L_b}\phi)-Z_{\delta_n}(P_{t}^{L_b}\phi)\right|{\cal F}_{\delta_n}^X		\right)=0$ $\mathbb{P}^n$-a.s..		Hence, $\dis \mathbb{P}^n(X_t(\phi_t))\leq \int P_{t}^{L_b}\phi(\delta_n,y)\mathbb{P}^n(u(\delta_n,y))dy\leq \int P_{t}^{L_b}\phi(0,y)u_0(y)dy.$ 
By inductively, the statement follows for all $t\in[0,\infty)$. Also, since $|X_t(\phi_t)|\leq X_t(|\phi_t|)$, the statement holds for $\phi\in {\cal D}_b(A)$.
\end{proof}

For $\lambda>0$, $t>0$, $q>0$, let \begin{align*}
\nu_n(\lambda,q,t)=\sup_{s\leq t}\int e^{\lambda |x|}\mathbb{P}^n\left(		u(s,x)^q		\right)dx.
\end{align*}

\begin{lem}\label{martindesc}\begin{enumerate}[(i)]
\item For $\phi\in {\cal D}_{b}(A)$, we have that\begin{align*}
\mathbb{P}^n\left(	\langle		Z(\phi)		\rangle_t		\right)<\infty.
\end{align*}
\item 
$\sup_n\nu_n(\lambda,q,t)<\infty$ for any $q>0.$
\end{enumerate}\end{lem}
\begin{proof}We will divide the proof into several steps as follows.
\begin{enumerate}[Step (1-$k$)]
\item \label{step1}For each $k\delta_n$ ($k\in \mathbb{N}$), $\mathbb{P}^n(\langle Z(\phi)\rangle_{k\delta_n})<\infty$.
\item \label{step2}We will extend $Z$ to an orthogonal martingale measure $\{Z(\phi):\phi\in {\cal E}_b\}$ (see \cite{Wal}) up to time $t=k\delta_n$.
\item \label{step3}$\sup_{n}\nu_n(\lambda,q,k\delta_n)<\infty.$
\end{enumerate}
We will prove the above statements by induction; step $(1$-$1)\Rightarrow$ step $(2$-$1)\Rightarrow$ step $(3$-$1)\Rightarrow $ step $(1$-$2)\Rightarrow\cdots$.

\textit{Step \ref{step1}}

It follows by the property of super-Brownian motion that
\begin{align*}
\mathbb{P}^n(\langle	Z(\phi)		\rangle_{\delta_n})&=\int \left(\gamma(u_0(y))\phi(s,y)^2\right)\mathbb{P}^n(u(s,y))dy\\
&=C(\phi)\int_0^{\delta_n}\mathbb{P}^n\left(X_s\left(\gamma(u_0(\cdot))\right)\right)ds\\
&\leq C(\phi)\int_0^{\delta_n}\int		P_{s}^{L_b}\gamma(u_0)(y)u_0(y)dyds	\\
&\leq C(\gamma,\phi)e^{L_b\delta_n}\int_0^{\delta_n}\int \left(1+P_{s}u_0(y)\right)u_0(y)dyds\\
&\leq C(\gamma,\phi)e^{L_b\delta_n}\left(\int_0^{\delta_n}\int \left(u_0(y)+u_0(y)^2\right)dyds 		\right)
\end{align*}
where we have used Lemma \ref{semigroup} in the third line.
Also, if $\nu_n(\lambda,q,k\delta_n)<\infty$, then we have by the Markov property that \begin{align*}
&\mathbb{P}^n\left(\left.\langle		Z(\phi)	\rangle_{(k+1)\delta_n}-\langle Z(\phi)\rangle_{k\delta_n}\right|		{\cal F}^X_{k\delta_n}	\right)\\
&=\int_{k\delta_n}^{(k+1)\delta_n}		\gamma (u(k\delta_n,y))\phi(s,y)^2 \mathbb{P}^n\left(\left.	u(s,y)\right|		u(k\delta_n,\cdot)	\right)dy		ds						\\
&\leq C(\phi)\int_{k\delta_n}^{(k+1)\delta_n}\int		\gamma(u^n(k\delta_n,y))\mathbb{P}_{u_{k\delta_n},b(u_{k\delta_n}),\gamma(u_{k\delta_n})}(u(s-k\delta_n,y))dyds\\
&\leq C(\gamma, \phi)e^{L_b\delta_n}\int_{k\delta_n}^{(k+1)\delta_n}\int		(1+P_{s}u^n(k\delta_n,y))u^n(k\delta_n,y)dyds\ \\
&\leq C(\gamma,\phi)e^{L_b\delta_n}\int_{k\delta_n}^{(k+1)\delta_n}\int \left(u^n(k\delta_n,y)+u^n(k\delta_n,y)^2\right)dyds,\ \ \mathbb{P}^n\text{-a.s., }
\end{align*}
where we have used the same argument of Proof of Lemma \ref{semigroup} in the forth line.
By taking expectation with assumption, we can obtain that $\mathbb{P}^n(\langle	Z(\phi)	\rangle_{(k+1)\delta_n})<\infty$.

\textit{Step \ref{step2}}

We assume that step (\ref{step1}-$k$).
We will show that we can extend $Z_t(\phi)$ be an orthogonal martingale on ${\mathbb{R}}$ up to $t=k\delta_n$. $\phi_m\in {\cal D}_b(A)$ and $\phi\in {\cal E}_b$ satisfies that  $\phi_n\to \phi$ pointwise boundedly, then \begin{align*}
\mathbb{P}^n\left(		\langle	Z(\phi_m)-Z(\phi_\ell)		\rangle_t		\right)&=\mathbb{P}^n\left(		\int_0^t\int		\gamma(u^n(s,y))(\phi_m(s,y)-\phi_\ell(s,y))^2u(s,y)dyds	\right)\\
&\to0,\ \ \text{as }m,\ell\to \infty
\end{align*} 
for $t\leq k\delta_n$ by the dominated convergence theorem. Thus, $Z_t(\phi_m)$ converges to a continuous square integrable martingale uniformly in $t$ on compacts in $L^2$ and we can extend $Z$ to an orthogonal martingale measure $\{Z(\phi):\phi\in {\cal E}_b\}$ such that\begin{align*}
\langle		Z(\phi)	\rangle_t=\int_0^t\int \gamma(u^n(s,y))\phi(s,y)^2u(s,y)dyds,\  \ \ \text{for }0\leq t\leq k\delta_n\ \text{and }\phi\in{\cal E}_b.
\end{align*}
Also, if $\phi:\Omega_X\times [0,k\delta_n]\times \mathbb{R}\to \mathbb{R}$ is ${\cal F}^X_{t+}\times {\cal B}(\mathbb{R}_+\times \mathbb{R})$-predictable and \begin{align*}
\mathbb{P}^n\left(		\int_0^t\int		\gamma(u^n(s,y))\phi(\omega,s,y)^2u(s,y)dyds		\right)<\infty,\ \ \text{for }{0\leq t\leq k\delta_n},
\end{align*}
then we can extend $Z$ to a stochastic integral of the form \begin{align*}
Z_t(\phi)=\int_0^t\phi(\omega,r,y)dZ(r,y).
\end{align*}

\textit{Step \ref{step3}}

The statement is true for the case $q=1$ by Lemma \ref{semigroup} for any $t>0$.  Indeed, let $\{\phi_m(x):m\in\mathbb{N}\}$ be the ${\cal D}_{b}(A)_+$-valued non-decreasing sequence such that $\lim_{m\to \infty }\phi_m(x)=e^{\lambda|x|}$ pointwisely.			Then, by Fubini's theorem and by monotone convergence theorem, we have that  
\begin{align*}
\mathbb{P}^n\left(\int	e^{\lambda|\cdot|}u(t,x)	dx	\right)&=\lim_{\ell\to \infty}		\mathbb{P}^n\left(\int\phi_\ell(x)u(t,x)dx\right)\\
&= \lim_{\ell\to\infty}\mathbb{P}^n\left(	X_t(\phi_\ell)		\right)\\
&\leq \limsup_{\ell\to \infty}\int P_{t}^{L_b}|\phi_\ell|(y)u_0(y)dy\\
&=\int P_{t}^{L_b}(e^{\lambda|\cdot|})(y)u_0(y)dy.
\end{align*}
By Lemma \ref{semigroup},
\begin{align*}
\nu_n(\lambda,1,t)&\leq \sup_{s\leq t}\int P_{s}^{L_b}(e^{\lambda |\cdot|})(y)u_0(y)dy\\
&\leq \sup_{s\leq t}e^{L_bs}\int P_{s}(e^{\lambda |\cdot|})(y)u_0(y)dy\\
&\leq c(t,\lambda)\int e^{\lambda |y|}u_0(y)dy,
\end{align*}
where we have used Lemma 6.2 in \cite{Shi} in the last line.
Thus, $\sup_n\nu_n(\lambda,1,t)<\infty$.

For $x_0\in \mathbb{R}$, let $\phi_s(x)=\frac{1}{\sqrt{2\pi s}}\exp\left(-\frac{(x-x_0)^2}{2s}\right)$ and for $\lambda'>\lambda\geq 0$, we define\begin{align*}
T(\ell)=\inf\left\{t\geq 0:|u_t|_{\lambda'}=\sup_{x}e^{\lambda' |x|}|u(t,x)|>\ell\right\}.
\end{align*}

We assume that step (\ref{step2}-$k$).
Then, we have for $0\leq s \leq  t\leq k\delta_n$  and $\ve>0 $ that    by taking $\psi(s,x)=e^{L_b(t-s)}\phi_{t+\ve-s}(x)$ \begin{align*}
&1\{T(\ell)\geq t\}\int \phi_\ve(x)u(t,x)dx\\
&\hspace{3em}+1\{T(\ell)<t\}\int		e^{L_b(t-T(\ell))}\phi_{t-T(\ell)+\ve}(y)u(T(\ell),y)dy\\
&= \int e^{L_bt}\phi_{t+\ve}(y)u_0(y)dy+ \int_0^{t\wedge T(\ell)} \int (b(u^n(s,y))-L_b)e^{L_b(t-s)} \phi_{t-s+\ve}(y)u(s,y)dyds\\
&\hspace{1em}+\int_0^{t\wedge T(\ell)}\int e^{L_b(t-s)}\phi_{t-s+\ve}(y)dZ(s,y).
\end{align*}
It is clear that each term converges $\mathbb{P}^n$-a.s.\,except  the last term as $\ve\searrow 0$. Therefore, \begin{align*}
&1\{T(\ell)\geq t\}u(t,x_0)\\
&\hspace{3em}+1\{T(\ell)<t\}\int		e^{L_b(t-T(\ell))}\phi_{t-T(\ell)}(y)u(T(\ell),y)dy\\
&= P_{t}^{L_b}(u_0)(x_0)+ \int_0^{t\wedge T(\ell)} \int (b(u^n(s,y))-L_b)e^{L_b(t-s)} \phi_{t-s}(y)u(s,y)dyds\\
&\hspace{1em}+\lim_{\ve\to 0}\int_0^{t\wedge T(\ell)}\int e^{L_b(t-s)} \phi_{t-s+\ve}(y)dZ(s,y),\ \ \mathbb{P}^n\text{-a.s.}
\end{align*}
Also,\begin{align*}
&\mathbb{P}^n\left(	\left(\int_0^{t\wedge T(\ell)}\int	e^{L_b(t-s)}\left(	\phi_{t-s+\ve}(y)-\phi_{t-s}(y)	\right)dZ(s,y)	\right)^2		\right)\\
&=\mathbb{P}^n\left(	\int_0^{t\wedge T(\ell)}\int	e^{2L_b(t-s)}\left(	\phi_{t-s+\ve}(y)-\phi_{t-s}(y)	\right)^2\gamma(u^n(s,y))u(s,y)dyds			\right)\\
&\leq c(\gamma)\mathbb{P}^n\left(		\int_0^{t\wedge T(\ell)}\int			e^{2L_b(t-s)}\left(	\phi_{t-s+\ve}(y)-\phi_{t-s}(y)	\right)^2	\left(u^n(s,y)+1\right))u(s,y)dyds		\right)\\
&\leq c(\gamma)\ell(\ell+1)\int_0^t\int e^{2L_b(t-s)}\left(	\phi_{t-s+\ve}(y)-\phi_{t-s}(y)	\right)^2dyds\to 0,\ \text{as }\ve\to0
\end{align*}
by the dominated convergence theorem.
Thus, we have that \begin{align*}
&1\{T(\ell)\geq t\}u(t,x_0)+1\{T(\ell)<t\}	P_{t-T(\ell)}^{L_b}u(T(\ell),\cdot)(x_0)\\
&= P_{t}^{L_b}u_0(x_0)+ \int_0^{t\wedge T(\ell)} \int (b(u^n(s,y))-L_b)e^{L_b(t-s)}\phi_{t-s}(y)u(s,y)dyds\\
&\hspace{1em}+\int_0^{t\wedge T(\ell)}\int e^{L_b(t-s)} \phi_{t-s}(y)dZ(s,y),\ \ \mathbb{P}^n\text{-a.s.}
\end{align*}
Especially, we have that \begin{align*}
u(t,x_0)1\{t\leq T(\ell)\}\leq  P_{t}^{L_b}u_0(x_0)+\int_0^{t\wedge T(\ell)}\int e^{L_b(t-s)} \phi_{t-s}(y)dZ(s,y),\ \ \mathbb{P}^n\text{-a.s.}
\end{align*}

Thus, by the Burkholder-Davis-Gundy inequality, \begin{align*}
&\mathbb{P}^n\left(	u(t,x_0)^q1\{t\leq T(\ell)\}		\right)\\
&\leq c(q)\left.\left(P_{t}^{L_b}u_0(x_0)\right)^q\right.\\
&\left.+c(q)\mathbb{P}^n\left(	\left(	\int_0^t\int		1\{t\leq T(\ell)\}e^{2L_b(t-s)}\phi_{t-s}(y)^2\gamma(u^n(s,y))u(s,y)dyds			\right)^{\frac{q}{2}}\right)				\right.\\
&\leq c(q,\gamma)\left(P_{t}^{L_b}u_0(x_0)\right)^q\\
&\left.+c(q,\gamma)\mathbb{P}^n\left(	\left(	\int_0^t\int		1\{t\leq T(\ell)\}e^{2L_b(t-s)}\phi_{t-s}(y)^2(u^n(s,y)+1)u(s,y)dyds			\right)^{\frac{q}{2}}\right)				\right.\\
&\leq c(q,\gamma)\left.\left(P_{t}^{L_b}u_0(x_0)\right)^q\right.\\
&\hspace{0em}+c(q,r)\left.	\left(	\int_0^t		\int\phi_{t-s}(y)^2dyds		\right)^{q/2-1}\right.\\
&\left.\times \mathbb{P}^n	\left(	\int_0^t\int		1\{t\leq T(\ell)\}e^{L_bq(t-s)}\phi_{t-s}(y)^2\left(u(s,y)^q+u^n(s,y)^q+u(s,y)^{q/2}\right)dyds			\right)\right)		\\
&\leq c(q,\gamma)\left(P_{t}^{L_b}u_0(x_0)^q+	\left(	t^{(q-2)/4}\int_0^t\int	(t-s)^{-1/2}	\phi_{t-s}(y)\right.\right.\\
&\hspace{5em}\left.\left.\mathbb{P}^n\left(1\{t\leq T(\ell)\}e^{L_bq(t-s)}\left(u(s,y)^q+u^n(s,y)^q+u(s,y)^{q/2}\right)		\right)dyds\right)\right),
\end{align*}
where we have used the facts that $\phi_{s}(y)^2\leq Ct^{-\frac{1}{2}}\phi(y)$ and $\int_0^t\int \phi(x)^2dxds\leq Ct^{1/2}$.
Let $\nu_n(\lambda,q,s,\ell)=\sup_{u\leq s}\int e^{\lambda|x|}\mathbb{P}^n\left(		u(u,x)^q	1\{s\leq T(\ell)\}	\right)dx$. Then, we have that \begin{align*}
&\nu_{n}(\lambda,q,t,\ell)\\
&\leq c(q,\gamma)\sup_{s\leq t}\left(		e^{L_bqt}\int P_{t}(e^{\lambda|\cdot|})(x)u_0(x)^qdx		\right.\\
&\hspace{3em}+t^{(q-2)/4}\int_0^t\int (t-s)^{-1/2}P_{t-s}(e^{\lambda|\cdot|})(y)e^{L_bq(t-s)}\\
&\left.\hspace{5em}\times\mathbb{P}^n\left(1\{t\leq T(\ell)\}\left(	u^n(s,y)^q+u(s,y)^q+u(s,y)^{q/2}		\right)\right)dyds\right)\\
&\leq c(q,\lambda,t,\gamma,L_b)\\
&\times \left(		\int (e^{\lambda|x|})(x)u_0(x)^qdx		+\int_0^t (t-s)^{-1/2}\left(	\nu(\lambda,q,s,\ell)+\nu(\lambda,q/2,s,\ell)		\right)	ds\right)\\
&\leq c(q,\lambda,t,\gamma,L_b)\left(	|u_0^q|_{2\lambda}\lambda^{-1}+\int_0^t(t-s)^{-1/2}(\nu(\lambda,q,s,\ell)+\nu(\lambda,q/2,s,\ell))ds				\right)
\end{align*}
Since $\nu_n(\lambda,q,t,\ell)<\infty$ for $q>1$ by definition of $T(\ell)$ and $\sup_{n}\nu_n(\lambda,1,t,\ell)\leq \sup_n\nu_n(\lambda,1,t)<\infty$, we have by using the Lemma \ref{hol} inductively that for $q=2^m$\begin{align*}
\sup_n\nu_n(\lambda,q,t,\ell)\leq c(q,\lambda,t,\gamma,L_b,u_0)<\infty,\ 0\leq t\leq k\delta_n.
\end{align*}
Thus, letting $\ell\to\infty$, we have $\sup_n\nu_n(\lambda,q,k\delta_n)<\infty$ for $q=2^m$ and also for any $q>0$.
\end{proof}

\begin{lem}\label{hol}(\cite[Lemma 4.1]{MuePer})
Assume $g:[0,T]\to [0,\infty)$ is bounded, $f:[0,T]\to [0,\infty)$ is non-decreasing, and $g(t)\leq c(f(t)+\int_0^t (t-s)^{-1/2}g(s)ds)$, $t\leq T$. Then, $g(t)\leq f(t)\exp(4ct^{1/2})$ for $t\leq T$. 
\end{lem}

\begin{cor}\label{cor1}
Let $\phi\in {\cal E}_b$. 
Then, we have that \begin{align*}
\int \phi(t,y)u(t,y)dy&=\int P_{t}^{L_b}\phi(0,y)u_0(y)dy\\
&\hspace{2em}+\int_0^t\int	(b(u^n)-L_b)P_{t}^{L_b}\phi(s,y)u(s,y)dyds\\
&\hspace{4em}+\int_0^{t}\int P_{t}^{L_b}\phi(s,y)dZ(s,y),
\end{align*}
and when we write $Z_t(\phi)=\int_0^{t}\int \phi(s,y)dZ(s,y)$, its quadratic variation is given by\begin{align*}
\int_0^t		\int \gamma (u^n(s,y))\phi^2(s,y)u(s,y)dyds,\ \ \mathbb{P}^n\text{-a.s.\,for all $t< \infty$}	
\end{align*}
if it is finite.
Especially, we have that for all $(t,x)\in \mathbb{R}_+\times \mathbb{R}$, \begin{align*}
u(t,x)&=P_{t}u_0(x)+\int_0^t\int b(u^n(s,y))\phi_{t-s}^x(y)u(s,\cdot)(y)dyds\\
&+\int_0^{t}\int \phi_{t-s}^x(y)dZ(s,y), \end{align*}
\text{where $\phi_{t}^x(y)=\frac{1}{\sqrt{2\pi t}}\exp\left(-\frac{(y-x)^2}{2t}\right)$.}
\end{cor}

The following two lemmas will be used to prove the tightness of $\mathbb{P}^n(u\in \cdot)$. 
\begin{lem}
Let $p_t(x)=\frac{1}{\sqrt{2\pi t}}\exp\left(-\frac{x^2}{2t}\right)$ and $p_t(x)=0$ if $t<0$.
If $T,\lambda>0$, there exists a $C(T,\lambda)<\infty$ such that \begin{align*}
&\int_0^t \int (p_{t-s}(y-x)-p_{t'-s}(y-x'))^2e^{-\lambda |y|}dyds\\
&\hspace{1em}\leq C(T,\lambda)(|x-x'|+|t-t'|^{1/2})e^{-\lambda |x|},\ \ \text{for }0< t< t'\leq T,|x-x'|\leq 1,\lambda>0.
\end{align*}
\end{lem}

\begin{lem}
Let $\{X_n(t,\cdot):t\geq 0,n\in\mathbb{N}\}$ be a sequence of continuous $C_{rap}(\mathbb{R})_+$-valued processes. Suppose that there exist some $\alpha>0$, $\beta>2$ and for all $T,\lambda>0$, there exists $C(T,\lambda)>0$ such that \begin{align*}
\mathbb{P}\left(	\left|	X_n(t,x)-X_n(t',x')	\right|^\alpha	\right)\leq C\left(		|x-x'|^\beta+|t-t'|^\beta		\right)e^{-\lambda|x|},\\
\ \ \text{for all }t,t'\in [0,T],\ |x-x'|\leq 1,n\in\mathbb{N}.
\end{align*}
If $\{\mathbb{P}_{X_n(0)}\}$ is tight on $C_{rap}(\mathbb{R})_+$, then $\{\mathbb{P}\left(X_n:n\in \mathbb{N}		\right)\}$ is tight on $C([0,\infty),C_{rap}(\mathbb{R})_+)$.

\end{lem}

The reader can refer the proof of these lemmas to \cite[Lemma 4.3 and Lemma 4.4]{MuePer}. 

\vspace{1em}

Now, we will show the tightness of $\mathbb{P}^n(u\in \cdot)$.

\begin{prop}\label{tightness}
$\{\mathbb{P}^n(u\in \cdot):n\in\mathbb{N}\}$ is tight on $C([0,\infty),C_{rap}^+)$.
\end{prop}
\begin{proof} Let $\hat{u}(t,x)=u(t,x)-P_{t}^{L_b}u_0(x)$.
Since $t\mapsto P_{t}^{L_b}u_0\in C([0,\infty),C_{rap}^+)$, we will show that $\{\mathbb{P}^n(\hat{u}\in \cdot):n\in\mathbb{N}\}$ is tight. Let $q\geq 1$, $0\leq t\leq t'\leq T$ and $|x-x'|\leq 1$ and $p_{t}(x)=\frac{1}{\sqrt{2\pi t}}\exp\left(	-\frac{x^2}{2t}		\right)$ $t\geq 0$ and $x\in\mathbb{R}$ and $p_t(x)=0$ for $t<0$. Then, by Corollary \ref{cor1}
\begin{align*}
&\mathbb{P}^n\left(		|\hat{u}(t,x)-\hat{u}(t',x')|^{2q}	\right)\\
&\leq c(q)\mathbb{P}^n	\left(	\left(\int_0^t\int		e^{L_b(t-s)}\left(	b(u^n(s,y))-L_b		\right)\left(p_{t-s}(y-x)-p_{t'-s}(y-x')\right)	u(s,y)dyds\right)^{2q}\right)\\
&+c(q)\mathbb{P}^n	\left(	\left(\int_0^{t}\int	\left(	e^{L_b(t'-s)}-e^{L_b(t-s)}\right)\left(	b(u^n(s,y))-L_b		\right)p_{t'-s}(y-x')	u(s,y)dyds\right)^{2q}\right)\\
&+c(q)\mathbb{P}^n	\left(	\left(\int_t^{t'}\int				e^{L_b(t'-s)}\left(	b(u^n(s,y))-L_b		\right)p_{t'-s}(y-x')	u(s,y)dyds\right)^{2q}\right)\\
&+	c(q)\mathbb{P}^n\left(\left(\int_0^t\int e^{2L_b(t-s)}\gamma(u^n(s,y))\left(p_{t-s}(y-x)-p_{t'-s}(y-x')\right)^2	u(s,y)dyds	\right)^q	\right)\\
&+	c(q)\mathbb{P}^n\left(\left(\int_0^t\int \left(e^{L_b(t'-s)}-e^{L_b(t-s)}\right)^2\gamma(u^n(s,y))\left(p_{t'-s}(y-x')\right)	^2u(s,y)dyds	\right)^q	\right)\\
&+	c(q)\mathbb{P}^n\left(\left(\int_t^{t'}\int e^{2L_b(t'-s)}\gamma(u^n(s,y))\left(p_{t'-s}(y-x')\right)^2	u(s,y)dyds	\right)^q	\right)\\
&\leq c(q,L_b,T,b)\left(	\int_0^t\int \left(\mathbb{P}^n\left(\left(	u^n(s,y)^\theta+1	\right)u(s,y)\right)^{2q}\right)e^{2q\lambda|y|}\left(p_{t-s}(y-x)-p_{t'-s}(y-x')\right)^2dyds		\right.\\
&\left.\times\left(\int_0^t\int	\left(p_{t-s}(y-x)-p_{t'-s}(y-x')\right)^2e^{-\frac{2q\lambda}{2q-1}|y|}dyds		\right)^{2q-1}\right)\\
&+c(q,L_b,T,b)|t'-t|^{2q}\left(	\int_0^t\int \left(\mathbb{P}^n\left(\left(	u^n(s,y)^\theta+1	\right)u(s,y)\right)^{2q}\right)e^{2q\lambda|y|}\left(p_{t'-s}(y-x')\right)^2dyds		\right.\\
&\left.\times\left(\int_0^t\int	\left(p_{t'-s}(y-x')\right)^2e^{-\frac{2q\lambda}{2q-1}|y|}dyds		\right)^{2q-1}\right)\\						
&+c(q,L_b,T,b)\left(	\int_t^{t'}\int \left(\mathbb{P}^n\left(\left(	u^n(s,y)^\theta+1	\right)u(s,y)\right)^{2q}\right)e^{2q\lambda|y|}\left(p_{t'-s}(y-x')\right)^2dyds		\right.\\
&\left.\times\left(\int_t^{t'}\int	\left(p_{t'-s}(y-x')\right)^2e^{-\frac{2q\lambda}{2q-1}|y|}dyds		\right)^{2q-1}\right)\\						
&+c(q,L_b,T,\gamma)\left(	\int_0^t\int	\mathbb{P}^n\left(\left(	\left(u^n(s,y)+1\right)u(s,y)		\right)^{q}\right)e^{q\lambda|y|}\left(p_{t-s}(y-x)-p_{t'-s}(y-x')\right)^2dyds				\right)\\
&\left. \times \left(		\int_0^t\int	\left(p_{t-s}(y-x)-p_{t'-s}(y-x')\right)^2e^{-\frac{q\lambda}{q-1}|y|}dyds			\right)^{q-1}\right)\\
&+c(q,L_b,T,\gamma)|t'-t|^{2q}\left(	\int_0^t\int	\mathbb{P}^n\left(\left(	\left(u^n(s,y)+1\right)u(s,y)		\right)^{q}\right)e^{q\lambda|y|}\left(p_{t'-s}(y-x')\right)^2dyds				\right)\\
&\left. \times \left(		\int_0^t\int	\left(p_{t'-s}(y-x')\right)^2e^{-\frac{q\lambda}{q-1}|y|}dyds			\right)^{q-1}\right)\\
&+c(q,L_b,T,\gamma)\left(	\int_t^{t'}\int	\mathbb{P}^n\left(\left(	\left(u^n(s,y)+1\right)u(s,y)		\right)^{q}\right)e^{q\lambda|y|}\left( p_{t'-s}(y-x')\right)^2dyds				\right)\\
&\left. \times \left(		\int_0^t\int	\left(p_{t'-s}(y-x')\right)^2e^{-\frac{q\lambda}{q-1}|y|}dyds			\right)^{q-1}\right)\\
&\leq c(q,\lambda,L_b,T,b,\gamma)\left(|x-x'|^{q-1}+|t-t'|^{\frac{q-1}{2}}\right)e^{-\lambda |x|}. 
\end{align*}
Since $u_0\in C_{rap}^+$, the result follows. 
\end{proof}

\begin{proof}[Proof of Theorem \ref{main} for $\frac{1}{2}<r\leq 1$]

For fixed $u_0(\cdot)\in C_{rap}^+(\mathbb{R})$, let $\left\{u_n(s,x):(s,x)\in\mathbb{R}_+\times \mathbb{R}\right\}$ have the law of $\mathbb{P}^n$.
By Lemma \ref{tightness}, there exist subsequences $N_k$ such that $u_{N_k}\Rightarrow u$ in $C([0,\infty),C_{rap}^+(\mathbb{R}))$.
By Skorohod's theorem, we may assume that $\{u_{N_k}:k\in\mathbb{N}\}\cup \{u\}$ can be constructed on the same probability space $(\Omega',{\cal F},\mathbb{P})$ and $u_{N_{k_{\ell}}}$ converges uniformly to $u$ $\mathbb{P}$-a.s.\,for a subsequence  $\{k_{\ell}:\ell\in \mathbb{N}\}$.

Then, we remark by continuity of $b$ and $\gamma$ that for any $T>0$ and $K>0$\begin{align*}
&\sup_{t\leq T}\sup_{|x|\leq K}\left|b(u^n(t,x))-b(u(t,x))\right|\to 0\\ 
&\sup_{t\leq T}\sup_{|x|\leq K}\left|\gamma(u^n(t,x))-\gamma(u(t,x))			\right|\to 0, \text{ as }n\to \infty
\end{align*}
$\mathbb{P}^n$-a.s. 
Lemma \ref{tightness} implies that for each $\phi\in C_{c}^{1,2}(\mathbb{R}_+\times\mathbb{R})$\begin{align}
&Z_t(\phi)=\int \phi(t,x)u(t,x)dx-\int \phi(0,x)u_0(x)dx-\int_0^t\int \left(A \phi(s,x)+b(u(s,x))\phi(s,x)\right)u(s,x)dxds,\notag\\
&Z_t(\phi)^2-\int_0^t \int \phi^2(s,x)\gamma(u(s,x))u(s,x)dxds\notag
\end{align}
are ${\cal F}_t^u$-martingales under $\mathbb{P}$. Then, we can extend $Z$ to an orthogonal martingale measure $\{Z(\phi):\phi\in{\cal E}_b\}$ by the same argument in the proof of Lemma \ref{martindesc}.

Thus, $u_t$ is a solution of the martingale problem (\ref{martiag}).


\end{proof}

In the end of this paper, we will give the proof of Theorem \ref{main} for $0<r\leq \frac{1}{2}$.

\begin{proof}[Proof of Theorem \ref{main} for $0<r\leq \frac{1}{2}$] Let $\sigma_n(u)={\sigma(u)}\left(\frac{u}{u+{1}/{n}}\right)^{1/2}$. Then, it is clear that $\sigma_n(u)$ satisfies that \begin{align*}
0\leq \sigma_n(u)\leq L'_\sigma (u^{r+1/2}+u)
\intertext{and}
\sup_n\sigma_n(u)=\sigma(u).
\end{align*}
Thus, for each $n$, there exists a $C_{rap}(\mathbb{R})_+$-valued  solution to the martingale problem (\ref{martiag}) for $b$ and $\sigma_n$. We denote it by $u_n$.
Then, the same results as the above Lemmas and Propositions are true for $u_n$ so that $\left\{\mathbb{P}(u_n\in \cdot):n\in\mathbb{N}\right\}$ is tight on $C([0,\infty),C_{rap}^+)$. Also, the same argument as the proof for $1/2<r\leq 1$ does hold and we complete the proof.
\end{proof}


\end{document}